\documentclass[a4paper,14pt]{article} 
\usepackage[T2A]{fontenc}
\usepackage[utf8]{inputenc}
\usepackage[russian,english]{babel} 
\usepackage{amssymb,amsfonts,amsmath,mathtext,enumerate,float,amsthm} 
\usepackage[unicode,colorlinks=true,citecolor=black,linkcolor=black]{hyperref}
\usepackage{indentfirst} 
\usepackage[dvips]{graphicx} 
\graphicspath{{illustr/}}

\makeatletter
\renewcommand{\@biblabel}[1]{#1.} 
\makeatother 

\usepackage{geometry} 
\usepackage[backend=biber,style=gost-numeric,sorting=none]{biblatex}
\theoremstyle{plain}
\newtheorem{theorem}{Theorem}[section]

\newtheorem{lemma}[theorem]{Lemma}
\newtheorem{corollary}[theorem]{Corollary}

\theoremstyle{definition}
\newtheorem{definition}[theorem]{Definition}

\newtheorem{remark}[theorem]{Remark}

\begin{document}


\title{
	On diameter bounds for planar integral point sets in semi-general position
	\footnote{
		This work was carried out at Voronezh State University and supported by the Russian Science
		Foundation grant 19-11-00197.
	}
}

\author{
	N.N. Avdeev
	\footnote{nickkolok@mail.ru, avdeev@math.vsu.ru}
}

\maketitle

\paragraph{Abstract.}
A point set $M$ in the Euclidean plane is called a planar integral point set if all the distances between the
elements of $M$ are integers, and $M$ is not situated on a straight line.
A planar integral point set is called to be in semi-general position, if it does not contain collinear triples.
The existing lower bound for mininum diameter of planar integral point sets is linear.
We prove a new lower bound for mininum diameter of planar integral point sets in semi-general position
that is better than linear.

\section{Introduction}

An \textit{integral point set} in a plane is a point set $M$ such that all the usual (Euclidean) distances between the
points of $M$ are integers and $M$ is not situated on a straight line.
Every integral point set consists of a finite number of points~\cite{anning1945integral,erdos1945integral};
thus, we denote the set of all planar integral point sets of $n$ points by
$\mathfrak{M}(2,n)$ (using the notation in~\cite{our-vmmsh-2018})
and define the diameter of $M\in\mathfrak{M}(2,n)$ in the following natural way:
\begin{equation}
	\operatorname{diam} M = \max_{A,B\in M} |AB|
	,
\end{equation}
where $|AB|$ denotes the Euclidean distance.
The symbol $\# M$ will be used for cardinality of $M$, that is the number of points in $M$ in our case.

Since every integral point set can obviously be dilated to a set of larger diameter,
minimal possible diameters of sets of given cardinality are in the focus.
To be precise,
the following function was introduced~\cite{kurz2008bounds,kurz2008minimum}:
\begin{equation}
	d(2,n) = \min_{M\in\mathfrak{M}(2,n)} \operatorname{diam} M
	.
\end{equation}

It turned out to be very easy to construct a planar integral point set of $n$ points with $n-1$ collinear ones and one point out of the line
(so-called \textit{facher} sets);
the same holds for 2 points out of the line (we refer the reader to~\cite{antonov2008maximal}, where some of such sets are called \textit{crabs})
and even for 4 points out of the line~\cite{huff1948diophantine}.
For $9\leq n\leq 122$, the minimal possible diameter is achieved at a facher set~\cite{kurz2008bounds}.

\begin{definition}
	A set $M\in\mathfrak{M}(2,n)$ is called to be in \textit{semi-general position},
	if no three points of $M$ are collinear.
	The set of all planar integral point sets in semi-general position
	is denoted by $\overline{\mathfrak{M}}(2,n)$.
\end{definition}

Furthermore, the constructions of integral point sets in semi-general position of arbitrary cardinality
appeared~\cite{harborth1993upper};
such sets are situated on a circle.
Also, there is a sophisticated construction of a circular integral point set of arbitrary cardinality
that gives the possible numbers of odd integral distances
between points in the plane~\cite{piepmeyer1996maximum}.

\begin{definition}
	A set $M\in\overline{\mathfrak{M}}(2,n)$ is called to be in \textit{general position},
	if no four points of $M$ are concyclic.
	The set of all planar integral point sets in general position
	is denoted by $\dot{\mathfrak{M}}(2,n)$.
\end{definition}

It remains unknown if there are integral points sets in general position of arbitrary cardinality;
however, some sets $M\in \dot{\mathfrak{M}}(2,7)$ are known~\cite{kreisel2008heptagon,kurz2013constructing}.

The inequality
\begin{equation*}
	d(2,n) \leq \overline{d}(2,n) \leq \dot{d}(2,n)
	,
\end{equation*}
where
$
	\overline{d}(2,n) = \min_{M\in\overline{\mathfrak{M}}(2,n)} \operatorname{diam} M
$
and
$
	\dot{d}(2,n) = \min_{M\in\dot{\mathfrak{M}}(2,n)} \operatorname{diam} M
$,
is obvious; however, a more interesting relation holds:
\begin{equation*}
	c_1 n \leq d(2,n) \leq \overline{d}(2,n) \leq n^{c_2 \log \log n}
	.
\end{equation*}
The upper bound is presented in~\cite{harborth1993upper}.
The lower bound was firstly introduced in~\cite{solymosi2003note};
the largest known value for $c_1$ is $5/11$ for $n\geq 4$~\cite{my-pps-linear-bound-2019}.

There are some bounds for minimal diameter of planar integral point sets in some special positions.
Assuming that the planar integral point sets contains many collinear points,
the following result holds.
\begin{theorem}~\cite[Theorem 4]{kurz2008minimum}
	For $\delta > 0$, $\varepsilon > 0$, and $P\in\mathfrak{M}(2,n)$ with
	at least $n^\delta$ collinear points, there exists a $n_0 (\varepsilon)$
	such that for all $n \geq n_0 (\varepsilon)$ we have
	\begin{equation}
		\operatorname{diam} P \geq n^{\frac{\delta}{4 \log 2(1+\varepsilon)}\log \log n}
		.
	\end{equation}
\end{theorem}
For diameter bounds for circular sets, we refer the reader to~\cite{bat2018number}.

Particular cases of planar integral point sets are also discussed
in~\cite[\S 5.11]{brass2006research},~\cite[\S D20]{guy2013unsolved},~\cite{our-pmm-2018},~\cite{our-ped-2018}.
For generalizaton in higher dimensions and the appropriate bounds, see~\cite{kurz2005characteristic,nozaki2013lower}.

In the present paper we give a special bound for planar integral point sets in semi-general position.
The condition of semi-general position is important in the given proof.

\section{Preliminary results}

In this section, we give some lemmas which will be used for the proof.

\begin{lemma}
	\cite[Observation 1]{solymosi2003note}
	If a triangle $T$ has integer side-lengths $a \leq b \leq c$,
	then the minimal height $m$ of it is at least $\left(a - \frac{1}{4}\right)^{1/2}$.
\end{lemma}

\begin{definition}
	The part of a plane between two parallel straight lines with distance $\rho$ between the lines
	is called a strip of width $\rho$.
\end{definition}

\begin{lemma}
	\cite{smurov1998stripcoverings}
	If a triange $T$ with minimal height $\rho$ is situated in a strip,
	then the width of a strip is at least $\rho$.
\end{lemma}

\begin{corollary}
	\label{cor:solymosi_strip}
	If a triangle $T$ with integer side-lengths $a \leq b \leq c$ is situated in a strip,
	then the width of a strip is at least $\left(a - \frac{1}{4}\right)^{1/2}$.
\end{corollary}

\begin{lemma}
	\cite[Lemma 4]{our-vmmsh-2018};
	\cite[Lemma 2.4]{my-pps-linear-bound-2019}
	\label{lem:square_container}
	Let $M\in\mathfrak{M}(2,n)$, $\operatorname{diam} M = d$.
	Then $M$ is situated in a square of side length $d$.
\end{lemma}

\begin{definition}
	\cite[Definition 2.5]{my-pps-linear-bound-2019}
	A \textit{cross} for points $M_1$ and $M_2$, denoted by $cr(M_1,M_2)$, is the union of two straight lines:
	the line through $M_1$ and $M_2$,
	and the perpendicular bisector of line segment $M_1 M_2$.
\end{definition}

\begin{lemma}
	\cite[Theorem 3.10]{my-pps-linear-bound-2019}
	\label{lem:no_distance_one}
	Each set $M\in\mathfrak{M}(2,n)$
	such that for some $M_1,M_2 \in M$ equality $|M_1 M_2|=1$ holds,
	consists of $n-1$ points, including $M_1$ and $M_2$, on a straight line,
	and one point out of the line, on the perpendicular bisector of line segment $M_1 M_2$.
\end{lemma}

\begin{lemma}
	\label{lem:count_of_points_on_hyperbolas}
	Let $\{M_1, M_2, M_3, M_4\} \subset M\in\overline{\mathfrak{M}}(2,n)$
	(points $M_2$ and $M_3$ may coincide, other points may not), $n\geq 4$.
	Then $\# M \leq 4 \cdot |M_1 M_2| \cdot |M_3 M_4|$.
\end{lemma}

\begin{remark}
	Lemma~\ref{lem:count_of_points_on_hyperbolas} is one of the variations of~\cite{erdos1945integral}.
\end{remark}

\begin{proof}
	Each point $N\in M$ satisfies one of the following conditions:

	a) $N$ belongs to $cr(M_1,M_2)$~--- overall at most 4 points;

	b) $N$ belongs to $cr(M_3,M_4)$~--- overall at most 4 points;

	c) $N$ belongs to the intersection of one of $|M_1 M_2| - 1$ hyperbolas
	with one of $|M_3 M_4| - 1$ hyperbolas~--- overall at most $4 (|M_1 M_2| - 1)(|M_3 M_4| - 1)$ points;

	Due to Lemma~\ref{lem:no_distance_one} we have $|M_1 M_2| \geq 2$ and $|M_3 M_4| \geq 2$.
	Since
	\begin{multline}
		4 (|M_1 M_2| - 1)(|M_3 M_4| - 1) + 4 + 4
		=
		4 ( (|M_1 M_2| - 1)(|M_3 M_4| - 1) + 2)
		=
		\\=
		4 ( |M_1 M_2| \cdot |M_3 M_4| + 1 - |M_1 M_2| - |M_3 M_4| + 2)
		=
		\\=
		4 ( |M_1 M_2| \cdot |M_3 M_4| + (1 - |M_1 M_2|) + (2 - |M_3 M_4|))
		\leq
		4 |M_1 M_2| \cdot |M_3 M_4|
		,
	\end{multline}
	we are done.
\end{proof}

\section{The main result}

\begin{theorem}
	\label{thm:main_result}
	For every integer $n \geq 3$ we have
	\begin{equation}
		\overline{d}(2,n) \geq (n/5)^{5/4}
		.
	\end{equation}
\end{theorem}

\begin{proof}
	For $n = 3$ we have $\overline{d}(2,3) = 1$.
	Consider $M\in\overline{\mathfrak{M}}(2,n)$, $n \geq 4$, $\operatorname{diam} M = p$.

	Let us choose points $M_1, M_2, M_3, M_4 \in M$
	(points $M_2$ and $M_3$ may coincide, other points may not), such that
	\begin{equation}
		\min_{A, B \in M} |AB| = |M_1 M_2|
		,
	\end{equation}
	\begin{equation}
		\min_{A, B \in M \setminus \{M_1\}} |AB| = |M_3 M_4| = m
		.
	\end{equation}

	For $m \leq p^{2/5}$, Lemma~\ref{lem:count_of_points_on_hyperbolas} yields that
	\begin{equation}
		n \leq 4 \cdot |M_1 M_2| \cdot |M_3 M_4| \leq  4 p^{4/5}
		,
	\end{equation}
	or, that is the same,
	\begin{equation}
		\label{eq:hyperbolas_5_4}
		p \geq (n/4) ^ {5/4} > (n/5) ^ {5/4}
		.
	\end{equation}

	So, let us consider $m > p^{2/5}$.
	Then for any $A,B \in M\setminus\{M_1\}$ the inequality $|AB| > p^{2/5}$ holds.
	Due to Corollary~\ref{cor:solymosi_strip}, no three points of $M\setminus\{M_1\}$
	are located in a strip of width $p^{1/5} / 2$.

	Lemma~\ref{lem:square_container} yields that $M$ is situated in a square with side length $p$.
	Let us partition this square into $q$ strips, $2p^{4/5} \leq q < 2p^{4/5} + 1$, each of width at most $p^{1/5} / 2$.
	Every strip contains at most two points of  $M\setminus\{M_1\}$,
	thus
	\begin{equation}
		\label{eq:strips_4_5}
		n \leq 2(2p^{4/5} + 1) + 1
		= 4p^{4/5}+3
		\leq 5 p^{4/5}
		.
	\end{equation}
	The latter inequality holds because $\overline{d}(2,n) \geq 4$ for $n\geq 4$~~\cite{kurz2008minimum}
	and $4^{4/5}>3$.
	From the inequality~\eqref{eq:strips_4_5} one can easily derive that
	\begin{equation}
		\label{eq:strips_5_4}
		p \geq (n/5) ^ {5/4}
		.
	\end{equation}
\end{proof}

\begin{remark}
	The following result in known:
\end{remark}

\begin{lemma}
	\cite[Corollary 1]{solymosi2003note}
	For $H \in \overline {\mathfrak{M}}(2,n)$, the minimum
	distance in H is at least $n^{1/3}$.
\end{lemma}
Applying the same technique, one can easily derive that
\begin{equation}
	n \leq 3 \frac{\operatorname{diam} H }{n^{1/6}}
	,
\end{equation}
which leads to the bound
\begin{equation}
	\overline{d}(2,n) \geq c_3 n^{7/6}
	,
\end{equation}
which is less than the one from Theorem~\ref{thm:main_result}.

\section{Conclusion}
The presented bound is the first special lower bound for sets in semi-general position.
Thus, we did not accepted the challenge to make the constant in Theorem~\ref{thm:main_result} as large as possible,
in order to keep the ideas of the proof clear and understandable.
A more thorough research can be done in the future to enlarge the constant.
However, the upper and lower bounds are still not tight.

\section{Acknowledgements}
Author thanks Dr. Prof. E.M. Semenov for proofreading and valuable advice.

\printbibliography

\end{document}